\documentclass[12pt]{amsart}
\usepackage{epsfig,amsmath,amsfonts,amssymb,amsthm}
\usepackage[all]{xy}
\usepackage{amscd}
\usepackage{amsmath, pb-diagram}
\usepackage{enumerate}
\newtheorem{thm}{Theorem}[section]
\newtheorem{lem}[thm]{Lemma}
\newtheorem{prop}[thm]{Proposition}
\newtheorem{cor}[thm]{Corollary}
\newtheorem{defn}[thm]{Definition}

\newtheorem{rem}[thm]{Remark}
\newtheorem{example}[thm]{Example}

\begin{document}
\title[A Generalization of Injective\\ and Projective Complexes ]
{New version of A Generalization of Injective and Projective Complexes }
\author{Tah\.{i}re \" Ozen}
\author{Em\.{i}ne Y{\i}ld{\i}r{\i}m}
\address{Department of Mathematics, Abant \.{I}zzet Baysal University
\newline
G\"olk\"oy Kamp\"us\"u Bolu, Turkey} \email{ozen\_t@ibu.edu.tr}
\email{emineyyildirim@gmail.com}
\date{\today}
\subjclass{18G35} \keywords{Injective (Projective) Complex, Precover, Preenvelope.}
\pagenumbering{arabic}
\begin{abstract} Let $\mathcal{X}$ be a class of $R$-modules. In this paper, we investigate  \;$\mathcal{X}$-(f.g.)injective ((f.g.)projective) and DG-$\mathcal{X}$-injective (projective) complexes which are generalizations of injective (projective) and DG-injective (projective) complexes. We prove that some known results can be extended to the class of \;$\mathcal{X}$-(f.g.)injective ((f.g.)projective) and DG-$\mathcal{X}$-injective (projective) complexes for this general settings.
\end{abstract}
\maketitle\section{Introduction}
Throughout this paper, $R$ denotes an associative ring with identity and all modules are unitary. Let $\mathcal{X}$ be a class of R-modules. An $R$-module $E$ is called $\mathcal{X}$-injective (see \cite{1}), if $Ext^{1}(B/A,E)=0$\; for every module $B/A \in \mathcal{X}$ or equivalently if $E$ is injective with respect to every exact sequence \;$0\rightarrow A\rightarrow B\rightarrow B/A\rightarrow 0$ where $B/A \in \mathcal{X}$. Dually it can be defined an \;$\mathcal{X}$-projective module. In section 2, we define \;$\mathcal{X}$-(f.g.)injective, $\mathcal{X}$-(f.g.)projective, DG-$\mathcal{X}$-injective and  DG-$\mathcal{X}$-projective complexes which are generalizations of injective, projective, DG-injective and DG-projective complexes (see \cite{2}) where f.g. abbreviates finitely generated. By \cite{2} we know that $(\varepsilon$, DG-injective) is a hereditary cotorsion pair and this cotorsion pair has enough injectives and enough projectives where $\varepsilon$ is a class of all exact complexes. We give some sufficient conditions that \;$(\varepsilon_{\mathcal{X}}$,DG-$\mathcal{X}$-injective) is also a hereditary cotorsion pair where $\varepsilon_{\mathcal{X}}$ is a class of all exact complexes whose kernels  are in $\mathcal{X}$ and what  \cite{5} denotes \;$\widetilde{\mathcal{X}}$ and  calls \;$\mathcal{X}$-complexes. We prove that \;${\varepsilon_{\mathcal{X}}}^{\bot}(^{\bot}\varepsilon_{\mathcal{X}}) =$
DG-$\mathcal{X}$-injective(projective) if $\mathcal{X}$ is extension closed and moreover if $\mathcal{X}$ is the class of finitely presented modules and $(\mathcal{X},\mathcal{X^{\bot}})$ is a cotorsion pair, then $(\varepsilon_{\mathcal{X}}$,DG-$\mathcal{X}$-injective) is a cotorsion pair and at the same time if $\mathcal{X}$ is closed under taking kernels of epics, then $DG-\mathcal{X}-injective \cap \varepsilon=\varepsilon_{\mathcal{X}-injective}\subseteq C(\mathcal{X}-f.g.injective)$.

In the last section, we will investigate when a complex has an exact C($\mathcal{X}$-(f.g.)projective((f.g.)injective))-precover(preenvelope). We know that an injective (projective) complex is exact. Moreover we give some conditions that an $\mathcal{X}$-(f.g.)injective ((f.g.)projective) complex is exact. Since every complex has an injective and projective resolution, we can compute the right derived functors $Ext^{i}(X,Y)$\; of $Hom(-,-)$\; where \;$Hom(X,Y)$ is the set of all chain maps from X to Y. Moreover\\
$\mathcal{H}om(X,Y)$ is the complex defined by $\mathcal{H}om(X,Y)_{n}=\prod_{p+q=n}Hom($\\$X_{-p},Y_{q})$. (See for more details and the other definitions \cite{2},\cite{4},\cite{3}).
\section{$\mathcal{X}$-injective and $\mathcal{X}$-projective complexes}
We begin with giving generalized definitions.
\begin{defn} \label{1.1} Let $\mathcal{X}$ be a class of R-modules. A
complex $\mathcal{C}: \ldots \longrightarrow C^{n-1} \longrightarrow
C^{n} \longrightarrow C^{n+1} \longrightarrow \ldots$ is called an
$\mathcal{X^{*}}-(cochain)$ complex, if  $C^i \in \mathcal{X}$ for all
$i \in \mathbb{Z}$. A complex $\mathcal{C}: \ldots \longrightarrow
C_{n+1} \longrightarrow C_{n} \longrightarrow C_{n-1}
\longrightarrow \ldots$ is called an $\mathcal{X^{*}}-(chain)$ complex,
if $C_i \in \mathcal{X}$ for all $i\in \mathbb{Z}$. The class of all
$\mathcal{X^{*}}-complexes$ is denoted by
$C(\mathcal{X^{*}})$.
\end{defn}
We recall that a complex $I$ is a finitely generated complex if every $I^{n}$ 
is a finitely generated module and $I$ is a bounded complex. Then we can give the following definition.
\begin{defn} \label{1.3} A complex $\mathcal{C}$ is called an $\mathcal{X}$-(f.g.)injective complex, if $Ext^{1}(Y/X,C)= 0$ for every (f.g.)complex $Y/X \in C(\mathcal{X^{*}})$. Equivalently,  a complex $\mathcal{C}$ is an $\mathcal{X}$-(f.g.)injective complex if for any exact sequence $0\rightarrow X\rightarrow Y\rightarrow Y/X\rightarrow0$ with a (f.g.)complex $Y/X 
\in C(\mathcal{X^{*}})$, the sequence $Hom(Y,C)\rightarrow Hom(X,C)\rightarrow 0$ is exact.

Dually we can define an $\mathcal{X}$-(f.g.)projective complex. A complex \;$\mathcal{C}$ is called an $\mathcal{X}$-(f.g.)projective complex, if $Ext^{1}(C,X) = 0$ for every (f.g.) complex $X \in \mathcal{C(X^{*})}$, or equivalently a complex \;$\mathcal{C}$ is an $\mathcal{X}$-(f.g.)projective complex if for any exact sequence $0\rightarrow X\rightarrow A\rightarrow B\rightarrow0$ with a (f.g.)complex $X \in \mathcal{C(X^{*})}$, the sequence $Hom(C,A)\rightarrow Hom(C,B)\rightarrow 0$ is exact.
We denote the class of all \;$\mathcal{X}$-(f.g.)injective\\
(f.g.)projective) complexes by \;C($\mathcal{X}$-(f.g.)injective((f.g.)projective)).
\end{defn}

\begin{defn} \label{1.15} Let $\varepsilon$ be the class of exact complexes. Then
we can define $\varepsilon_{\mathcal{X}}$ such that $\varepsilon_{\mathcal{X}}$ is the
class of  exact complexes with kernels in $\mathcal{X}$.
\end{defn}

\begin{example} \label{1.4} If $P$ is an $\mathcal{X}-projective(\mathcal{X}-injective)$ module, then
$\overline{P}:... \longrightarrow 0 \longrightarrow P
\longrightarrow P \longrightarrow 0 \longrightarrow 0
\longrightarrow ...$ is an
$\mathcal{X}-(f.g.)projective(\mathcal{X}-(f.g.)injective)$ complex. Moreover any direct sum (product) of $\mathcal{X}-(f.g.)projective$ $(\mathcal{X}-(f.g.)injective)$ complexes is again $\mathcal{X}-(f.g.)projective(\mathcal{X}-(f.g.)injective)$. Since C($\mathcal{X}$-(f.g.)injec-
tive((f.g.)projective)) is closed under extensions, every bounded exact complex $Y:...0\rightarrow Y^{0}\rightarrow...\rightarrow Y^{n}\rightarrow 0...$ with kernels $\mathcal{X}$-injective(pro-
jective) module is in C($\mathcal{X}$-(f.g.)injective((f.g.)projective)).
  
Since every left (right) bounded exact complex with kernels $\mathcal{X}$-injecti-
ve(projective) module is an inverse (direct)limit of bounded exact complexes with kernels $\mathcal{X}$-(f.g.)injective(projective) module, then every left (right) bounded exact complex with kernels $\mathcal{X}$-injective(projective) module is in C($\mathcal{X}$-(f.g.)injective((f.g.)projective)).

 If $\mathcal{X}$ is a class of finitely presented modules, then since every exact complex with kernels $\mathcal{X}$-injective module is a direct limit of left bounded exact complex with kernels $\mathcal{X}$-injective module, every exact complex with kernels $\mathcal{X}$-injective module is in C($\mathcal{X}$-f.g. injective), that is $\varepsilon_{\mathcal{X}-injective}\subseteq C(\mathcal{X}-f.g. injective)$.
 
 Moreover if $\mathcal{X}-injective\subseteq \mathcal{X}$, then every $\epsilon_{\mathcal{X}-injective}$ complex is a direct sum of injective complexes the same as injective complexes and similarly if $\mathcal{X}-projective\subseteq \mathcal{X}$, then every $\epsilon_{\mathcal{X}-projective}$ complexes is a direct sum of projective complexes. Then 
$\varepsilon_{\mathcal{X}-injective(projective)}\subseteq C(\mathcal{X}-injective(projective))$.
\end{example}
 Notice that if $P$ is an
$\mathcal{X}$-injective($\mathcal{X}$-projective) module and $P$ is
not in the class $\mathcal{X}$, then $\overline{P}$ is an
$\mathcal{X}$-(f.g.)injective complex, but not an $\mathcal{X^{*}}$-complex.
So $\mathcal{X}$-(f.g.)injective ((f.g.)projective) complex may not be an
$\mathcal{X^{*}}$-complex.
 \begin{lem} \label{1.27} Let $X$ be an $\mathcal{X}-injective$ complex such that \;$\frac{E(X)}{X} \in C(\mathcal{X}^{*})$\; (or $\frac{Y}{X} \in C(\mathcal{X}^{*})$) where $E(X)$ is an injective
envelope of $X$. Then $X=E(X)$ and so it is an injective complex.(X is a
direct summand of Y).
\end{lem}

\begin{proof} We know that every complex has an injective envelope,
so $X$ has an injective envelope $E(X)$. Then $E(X)$ is an injective
complex, and so it is exact. We have the following commutative
diagram;
\[\begin{diagram}
\node{0} \arrow{e}  \node{X} \arrow{s,t} {id_x} \arrow{e,r} {i}
\node{E(X)} \arrow{sw,r,..} {\phi} \\
\node[2]{X}
\end{diagram}\]
\noindent such that ${\phi}{i}={id}_{x}$. Therefore $X$ is a direct
summand of $E(X)$. So $X$ is an injective complex and hence it is
exact. Similarly, if $\frac{Y}{X} \in C(\mathcal{X}^{*})$, then we can
prove that X is a direct summand of Y.
\end{proof}

\begin{defn} \label{1.5} A complex I is
called DG-$\mathcal{X}$-injective, if each $I^{n}$ is
$\mathcal{X}$-injective and $\mathcal{H}om(E,I)$ is exact for all $E
\in \varepsilon_{\mathcal{X}}$.
A complex I is called DG-$\mathcal{X}$-projective, if each $I^{n}$
is $\mathcal{X}$-projective and $\mathcal{H}om(I,E)$ is exact for
all $E \in \varepsilon_{\mathcal{X}}$.
\end{defn}

\begin{lem} \label{1.6} Let $A \overset{\beta} \longrightarrow B \overset{\theta} \longrightarrow
C$ be an exact sequence of modules (complexes) where $Ker \beta \in \mathcal{X}$(C($\mathcal{X}^{*}$)). Then for
all $\mathcal{X}$-projective modules (complexes) I, $Hom(I,A) \longrightarrow
Hom(I,B) \longrightarrow Hom(I,C)$ is exact.
\end{lem}

\begin{proof} By the exact sequence $0 \longrightarrow Ker{\theta} \overset{i} \longrightarrow B \overset{\theta} \longrightarrow
C$, $0 \longrightarrow Hom(I,Ker{\theta})
\longrightarrow Hom(I,B) \longrightarrow Hom(I,C)$ is exact.
We have the following commutative diagram;
\[\begin{diagram} \node{A} \arrow{e,t} {\beta} \node{Im{\beta}} \arrow{e}
\node{0} \\ \node[2]{I} \arrow{n,t} {g} \arrow{nw,r,..} {f} 
\end{diagram}\]
\noindent such that ${\beta}{f}={g}$.
\noindent Since I is $\mathcal{X}$-projective module (complex) and
$Ker \beta \in \mathcal{X}$ $(C(\mathcal{X^{*}}))$, Hom$(I,A) \longrightarrow $Hom$(I,B) \longrightarrow
$Hom$(I,C)$ is exact.
\end{proof}
\noindent Dually we can give the following lemma;
\begin{lem} \label{1.7} Let $A \overset{\beta} \longrightarrow B \overset{\theta} \longrightarrow
C$ be an exact sequence of modules (complexes) where $\frac{C}{Im{\theta}} \in
\mathcal{X}$(C($\mathcal{X}^{*}$)). Then for all $\mathcal{X}$-injective modules (complexes) I,
$Hom(C,I) \longrightarrow
Hom(B,I) \longrightarrow Hom(A,I)$ is exact.
\end{lem}
\begin{example} \label{1.8} Let $I=.... \longrightarrow 0 \longrightarrow I^{0} \longrightarrow
0 \longrightarrow 0 \longrightarrow ...$ where $I^{0}$ is an
$\mathcal{X}$-injective($\mathcal{X}-projective$) module. Then I is
DG-$\mathcal{X}$-injective (DG-$\mathcal{X}$-projective) complex.
\end{example}
\begin{proof} Let $E:... \longrightarrow E^{-1} \overset{d^{-1}} \longrightarrow E^{0} \overset{d^{0}} \longrightarrow
E^{1} \overset{d^{1}} \longrightarrow E^{2} \overset{d^{2}}
\longrightarrow E^{3} \longrightarrow...$ be exact and $Kerd^{n}
\in \mathcal{X}$, then
$\mathcal{H}om(E,I) \cong ... Hom(E^{2},I^{0}) \longrightarrow Hom(E^{1},I^{0})
\longrightarrow Hom(E^{0},I^{0}) ...$. By Lemma \ref {1.7} $\mathcal{H}om(E,I)$ is exact.
\end{proof}
\begin{lem} \label{1.9} If a complex $X: \ldots \longrightarrow X_{n+1} \longrightarrow X_{n} \longrightarrow
X_{n-1} \longrightarrow \ldots$ is an
$\mathcal{X}$-injective($\mathcal{X}$-projective) complex, then for
all $n \in \mathbb{Z}$ $X_{n}$ is an
$\mathcal{X}$-injective($\mathcal{X}$-projective) module. Moreover if $\mathcal{X}$ is a class of finitely generated modules and $X$ is an $\mathcal{X}$- f.g. injective (projective) complex, then for
all $n \in \mathbb{Z}$ $X_{n}$ is an
$\mathcal{X}$-injective($\mathcal{X}$-projective) module.
\end{lem}
\begin{proof} Let $0 \longrightarrow N \overset{i} \longrightarrow
M$ be exact such that $\frac{M}{N} \in \mathcal{X}$ and $\alpha: N \rightarrow
X_{n}$ be linear form the pushout;
\[\begin{diagram}
\node{N} \arrow{e,t}{i} \arrow{s,t}{\alpha} \node{M} \arrow{s,t,..}{\gamma_{n}}\\
\node{X_{n}} \arrow{e,t,..}{\theta_{n}} \node{\frac{X_{n} \oplus M}{A}} 
\end{diagram}\]
\noindent where $A=\{(\alpha(n),-i(n)): n \in N\}$.
By the following diagram;
\[\begin{diagram}
\node{0} \arrow{e} \node{X_{n + 1}} \arrow{e} \arrow{s} \node{X_{n +
1}} \arrow{e} \arrow{s} \node{0} \arrow{e} \arrow{s} \node{0} \\
\node{0} \arrow{e} \node{X_{n}} \arrow{e} \arrow{s} \node{\frac{{M
\oplus X_{n}}}{A}} \arrow{e} \arrow{s} \node{\frac{M}{N}} \arrow{e}
\arrow{s}  \node{0} \\
\node{0} \arrow{e} \node{X_{n - 1}} \arrow{e} \node{X_{n - 1}}
\arrow{e} \node{0} \arrow{e} \node{0}
\end{diagram}\]
\noindent we have the exact sequence $0 \longrightarrow X
\longrightarrow T \longrightarrow S\longrightarrow 0$ where
$T: ... \longrightarrow X_{n + 2} \longrightarrow  X_{n + 1}
\longrightarrow \frac{M \oplus X_{n}}{A} \longrightarrow X_{n - 1}
...$ and
$S: ... \longrightarrow 0 \longrightarrow  0 \longrightarrow \frac{M
}{N} \longrightarrow 0 ...$.
\noindent Since $X$ is a $\mathcal{X}-injective$ complex,
$Ext^{1}(S,X)$\\
$=0$, and so $0 \rightarrow Hom(S,X) \rightarrow Hom(T,X) \rightarrow Hom(X,X)
\rightarrow Ext^{1}(S,X)=0$. Therefore there exists $\beta_{n}:
T_{n}= \frac{M\oplus X_{n}}{A} \longrightarrow X_{n}$ such that
${\beta_{n}}{\theta_{n}}=1$. So
$${{\beta}^{n}}{\theta^{n}} (\alpha(n))=\alpha(n)$$
$${\beta}^{n}((\alpha(n),0)+A)=\alpha(n)$$
$${\beta}^{n}((0,i)+A)=\alpha(n)$$
$${{\beta}^{n}}{\gamma_{n}}i(n)=\alpha(n)$$
\noindent And hence ${{\beta}^{n}}{\gamma_{n}}i=\alpha$. So $X_{n}$
is an $\mathcal{X}$-injective module.
\end{proof}
\noindent The following example shows that if $X:... \rightarrow X_{n+1} \rightarrow X_{n} \rightarrow X_{n-1} \rightarrow ...$
is a complex such that $X_{n}$ are
$\mathcal{X}$-injective($\mathcal{X}$-projective) modules for all $n
\in \mathbb{Z}$, then X does not need to be \;$\mathcal{X}$-(f.g.)injective ($\mathcal{X}$-(f.g.)projective) complex.
\begin{example} \label{1.10} Let $R\in \mathcal{X}$ and $\mathcal{X}$-injective module and $f:R \rightarrow R\oplus R$ be a morphism such that $f(a)=(0,a)$ and $g:R \oplus R \rightarrow R$ be a morphism such
that $g(a,b)=a$. Then $gf=0$ where $g\neq 0$. Let we have the following diagrams;
\[\begin{diagram}
\node{...0} \arrow{e} \arrow{s} \node{R}
\arrow{e,t} {f} \arrow{s,r} {f} \node{R\oplus R} \arrow{e} \arrow{s,r} {g} \node{0...}  \arrow{s}  \\
\node{...0} \arrow{e} \node{0} \arrow{e} \node{R}
\arrow{e} \node{0...}  
\end{diagram}\]
\[\begin{diagram}
\node{...0} \arrow{e} \arrow{s} \node{R}
\arrow{e,t} {f} \arrow{s,r} {f} \node{R\oplus R} \arrow{e} \arrow{s,r} {1} \node{0...}  \arrow{s}  \\
\node{...0} \arrow{e} \node{R \oplus R}
\arrow{e,t} {1} \node{R \oplus R} \arrow{e} 
\node{0...} 
\end{diagram}\]

\noindent Then we have the diagram as follow,
\[\begin{diagram}
\node{...0} \arrow{e} \arrow{s} \node{R}
\arrow{e,t} {f} \arrow{s,r} {f} \node{R\oplus R} \arrow{e} \arrow{s,r} {1} \node{0...}  \arrow{s}  \\
\node{...0} \arrow{e} \arrow{s} \node{R \oplus R}
\arrow{e,t} {1} \arrow{s} \node{R \oplus R} \arrow{e} \arrow{s,r} {g}
\node{0...} 
\arrow{s} \\
\node{...0} \arrow{e} \node{0} \arrow{e} \node{R}
\arrow{e} \node{0...}  
\end{diagram}\]
\noindent such that $g1=0$. But this is impossible. So $\underline{R}$ cannot be an
$\mathcal{X}$-injective complex.
\noindent Dually, we can give an example for
$\mathcal{X}$-projectivity.
\end{example}

\begin{rem}There exists a module is both in $\mathcal{X}$ and an $\mathcal{X}$-injective module.
Let $\mathcal{X}$ be a class of injective modules and $R$ be an injective module, then $R$ is both in $\mathcal{X}$ and an $\mathcal{X}$-injective module. 	Moreover let a module $M$ be a flat cotorsion module (see Theorem 5.3.28 in \cite{4} for existence of such a module) and $\mathcal{X}$ be a class of flat modules, so is $M$.
\end{rem}
\begin{lem} \label{1.11} If $I \in \varepsilon_{\mathcal{X}}^\bot$, then each
$I^{n}$ is $\mathcal{X}$-injective for each $n \in \mathbb{Z}$.
\end{lem}
\begin{proof} Let $S \subseteq M$ be a submodule of a module M with $\frac{M}{S} \in \mathcal{X}$ and $\alpha: S
\longrightarrow I_{n}$ be linear form the pushout;
\[\begin{diagram}
\node{S} \arrow{e,t}{i} \arrow{s,t}{\alpha} \node{M} \arrow{s,t,..}{i_{1}}\\
\node{I^{n}} \arrow{e,t,..}{i_{2}} \node{\frac{I^{n}\oplus M}{A}=I^{n}\oplus_{S} M}
\end{diagram}\]
\noindent where $A=\{(\alpha(s),-s): s\in S\}$. Thus $i_{2}$ is
one-to-one the same as i. Then $\overline{I}:... \longrightarrow
{I^{n-1}} \longrightarrow {I^{n}\oplus_{S} M} \longrightarrow
{I^{n+1}} \longrightarrow {I^{n+2}} \longrightarrow ...$ is a
complex.
\[\begin{diagram}
\node{0} \arrow{e} \node{I^{n-1}} \arrow{e} \arrow{s} \node{I^{n-1}}
\arrow{e} \arrow{s} \node{0} \arrow{e} \arrow{s} \node{0} \\
\node{0} \arrow{e} \node{I^{n}} \arrow{e} \arrow{s} \node{I^{n}
\oplus_{S} M} \arrow{e} \arrow{s} \node{\frac{M}{S}} \arrow{e}
\arrow{s}  \node{0} \\
\node{0} \arrow{e} \node{I^{n+1}} \arrow{e} \node{I^{n+1}} \arrow{e}
\node{\frac{M}{S}} \arrow{e} \node{0} 
\end{diagram}\]
Therefore, we have an exact sequence $0 \longrightarrow I
\longrightarrow \overline{I} \longrightarrow E \longrightarrow 0$
where $E:...\longrightarrow \frac{M}{S} \longrightarrow \frac{M}{S}
\longrightarrow 0 \longrightarrow 0 \longrightarrow 0
\longrightarrow ...$ and so we have an exact sequence $0
\longrightarrow Hom(E,I) \longrightarrow Hom(\overline{I},I)
\longrightarrow Hom(I,I) \longrightarrow Ext^{1}(E,I)=0$ since $I \in
\varepsilon_{\mathcal{X}}^\bot$.
This implies that we can find $\overline{f}:\overline{I}
\longrightarrow I$ with $\overline{f}f=1$. Therefore, there exists a
function $\overline{f}^{n}:{I^{n} \oplus_{S} M} \longrightarrow
I^{n}$ with $\overline{f}^{n}f^{n}=1$. So,
$$\overline{f}^{n}f^{n} (\alpha(s))=\alpha(s)$$
$$\overline{f}^{n}((\alpha(s),0)+A)=\alpha(s)$$
$$\overline{f}^{n}((0,s)+A)=\alpha(s)$$
$$\overline{f}^{n}i_{1}i(s)=\alpha(s)$$
\noindent and hence ${\overline{f}_{n}}i_{1}i=\alpha$ and thus each
$I^{n} \in \mathcal{X}$-injective.

\end{proof}

\begin{lem} \label{1.12} Let $f:X \longrightarrow Y$ be a morphism of complexes.
Then the exact sequence $0 \longrightarrow Y \longrightarrow M(f)
\longrightarrow X[1] \longrightarrow 0$ associated with the mapping
cone $M(f)$ splits if and only if f is homotopic
to 0.
\end{lem}

\begin{proof} It follows from \cite{2}.
\end{proof}

\begin{lem} \label{1.13} Let $X$ and $I$ be complexes. If $Ext^{1}(X,I[n])=0$
for all $n \in \mathbb{Z}$, then $\mathcal{H}om(X,I)$ is exact.
\end{lem}

\begin{proof} Since $Ext^{1}(X,I[n])=0$, if $f: X[-1] \rightarrow
I[n]$ is a morphism, then $0 \rightarrow I[n] \rightarrow M(f)
\rightarrow X \rightarrow 0$ splits.

By Lemma \ref{1.12}, $f: X[-1] \rightarrow I[n]$ is homotopic to zero for
all n. So $f^{1}: X \rightarrow I[n+1]$ is homotopic to zero for all $n \in
\mathbb{Z}$. Thus $\mathcal{H}om(X,I)$ is exact.
\end{proof}
\noindent In \cite{5} the following theorem is proved in the case
when $(\mathcal{X},\mathcal{X}^{\bot})$ is a cotorsion pair.
\begin{thm}\label{t} Let $\mathcal{X}$ be extension closed. Then \;${\varepsilon_{\mathcal{X}}}^{\bot}(^{\bot}\varepsilon_{\mathcal{X}})$ =
DG-$\mathcal{X}$-injective(projective).
\end{thm}
\begin{proof} By Lemma 2.13 and Lemma 2.15 we have that \;${\varepsilon_{\mathcal{X}}}^{\bot}(^{\bot}\varepsilon_{\mathcal{X}}) \subseteq$
DG-$\mathcal{X}$-injective(projective).
\noindent Let $I\in$ DG-$\mathcal{X}$-injective. Then $\mathcal{H}om(X,I)$ is exact for all $X \in \varepsilon_{\mathcal{X}}$ and so for all n $f:X\rightarrow I[n]$ is homotopic to zero. By Lemma \ref{1.12} $A:0\rightarrow I[n]\rightarrow M(f)\rightarrow X[1]\rightarrow 0$ is split exact. We know that any exact complex \;$B:0\rightarrow I[n]\rightarrow Y\rightarrow X[1]\rightarrow 0$\; splits at module level since the $I[n]^{m}$ are $\mathcal{X}$-injective modules and $X^{m}\in \mathcal{X}$. Therefore the exact sequences $A$ and $B$ are isomorphic.
It is known that $Ext^{1}(C,A) = 0 $ if and only if every short exact sequence
$0 \rightarrow A \rightarrow B \rightarrow C \rightarrow 0$ is split. This implies that $Ext^{1}(X,I[n])=0$ and thus the converse inclusion is proved.  
\end{proof}
If we use Theorem \ref{t}, then we can give the following examples since $\mathcal{X}$ and ${\varepsilon_{\mathcal{X}}}^{\bot}(^{\bot}\varepsilon_{\mathcal{X}})$ 
are extension closed and every right(left) bounded complex is a direct (inverse) limit of bounded complexes.
\begin{example} Let $\mathcal{X}$ be extension closed. Then every $\mathcal{X}$-projective\\(injective) complex is DG-$\mathcal{X}$-projective(injective). Every right(left)\\ bounded complex I where $I_{i}$ is an
$\mathcal{X}$-projective(injective) module is a DG-$\mathcal{X}$-projective(injective) complex. 
\end{example}
\begin{prop}\label{c} Let $\mathcal{X}$ be closed under extensions, summands and direct limits. If $\mathcal{X}$ is a precovering class containing projective modules, then  $(\varepsilon_{\mathcal{X}}$,
DG-$\mathcal{X}$-injective) and (DG-$\mathcal{Y}$-projective,$\varepsilon_{\mathcal{Y}})$ are cotorsion pairs where $\mathcal{Y}=\mathcal{X}$-injective. If $\mathcal{X}$ is closed under taking kernels of epics and the class of finitely presented modules, then this cotorsion pairs are hereditary and DG-$\mathcal{X}$-injective$ \cap \varepsilon=\varepsilon_{\mathcal{X}-injective}\subseteq C(\mathcal{X}-f.g.injective)$. Moreover if $C(\mathcal{X^{*}})$ is also a covering class , then every complex has a monic C($\mathcal{X}-injective$) preenvelope with cokernel in $C(\mathcal{X^{*}})$.
\end{prop}
\begin{proof} Since $\mathcal{X}$ is a precovering class containing projective modules and it is closed under direct limits, it is an epic covering class . By Wakamatsu Lemma it is a special covering class, so $(\mathcal{X},\mathcal{X}^{\bot})$ is a complete cotorsion pair (see \cite{4}). If $\mathcal{X}$ is closed under taking kernels of epics and the class of finitely presented modules, then $(\varepsilon_{\mathcal{X}}$, DG-$\mathcal{X}$-injective) is a hereditary cotorsion pair and $DG-\mathcal{X}-injective\cap \varepsilon=\varepsilon_{\mathcal{X}-injective}\subseteq C(\mathcal{X}-f.g.injective)$ by Corollary 3.13 in \cite{5} and Example 2.4. The other parts are easy.
\end{proof}
\begin{cor} Let $\mathcal{X}$ be closed under extensions, direct sum and pure quotients and let $\mathcal{X}$ contain the ring $R$.  Then $(\mathcal{X},\mathcal{X}-injective)$ is a perfect cotorsion pair. Thus $(\varepsilon_{\mathcal{X}}$,DG-$\mathcal{X}$-injective) and (DG-$\mathcal{Y}$-projective,$\varepsilon_{\mathcal{Y}})$ are cotorsion pairs where $\mathcal{Y}=\mathcal{X}$-injective. 
\end{cor}
\begin{proof}By \cite{7},it is closed under direct limit and summands and thus it is covering class containing projective modules. Then $(\mathcal{X},\mathcal{X}-injective)$ is a complete cotorsion pair. Since $\mathcal{X}$ is closed under direct limit, it is a perfect cotorsion theory. By Proposition \ref{c} $(\varepsilon_{\mathcal{X}}$, DG-$\mathcal{X}$-injective) and (DG-$\mathcal{Y}$-projective,$\varepsilon_{\mathcal{Y}})$ are  cotorsion pairs. 
\end{proof}
\begin{cor}Let $\mathcal{X}$ be the class of finitely presented modules and $(\mathcal{X},\mathcal{X^{\bot}})$ be a cotorsion pair. Then $(\mathcal{X},\mathcal{X^{\bot}})$, $(\varepsilon_{\mathcal{X}}$, 
DG-$\mathcal{X}$-injective) and (DG-$\mathcal{Y}$-projective,$\varepsilon_{\mathcal{Y}})$ are cotorsion pairs where $\mathcal{Y}=\mathcal{X}$-injective.
\end{cor}
\begin{proof}Since $\mathcal{X}$ is the class of finitely presented modules, $\mathcal{X^{\bot}}=\mathcal{X}-injective$ is closed under pure submodules. Moreover it is closed under direct product and inverse limit. So  $\mathcal{X^{\bot}}$ is an enveloping class by \cite{2}. Since $(\mathcal{X},\mathcal{X^{\bot}})$ is a cotorsion pair, $\mathcal{X^{\bot}}$ is a special enveloping class with cokernel in $\mathcal{X}$. So again by \cite{2} $(\mathcal{X},\mathcal{X^{\bot}})$ is a complete cotorsion pair. 
\end{proof}

\section{C($\mathcal{X}$-(f.g.)projective)-precovers and C($\mathcal{X}$-(f.g.)injective)-preenvelopes }
In this section we prove that if a complex has a C($\mathcal{X}$-projective)-precover or C($\mathcal{X}$-injective)-preenvelope in C($\mathcal{X}^{*}$), then such precovers  or preenvelopes  are homotopic.  Moreover we investigate when a complex has an exact C($\mathcal{X}$-(f.g.)projective((f.g.)injective))-
precover(preen-\\velope) and we give a condition when an $\mathcal{X}$-(f.g.)projective((f.g.)injec-\\tive) complex is exact . 
\begin{lem} \label{1.20} Let $f:X \longrightarrow Y$ be a chain
morphism, $X$ is an $\mathcal{X^{*}}$ complex and $Y$ is an
$\mathcal{X}-injective$ complex. Then $f$ is homotopic to zero. Moreover if a complex 
has a C($\mathcal{X}$-injective)-preenvelope in $C(\mathcal{X^{*}})$, then such preenvelopes are homotopic.
\end{lem}
\begin{proof}Let $id:X \longrightarrow X$, then we have the
following exact sequence;
\[\begin{diagram}
\node{0} \arrow{e} \node{X} \arrow{e,t} {i} \arrow{s,r} {f}
\node{M(id)} \arrow{e} \arrow{sw,r,..} {g} \node{X[1]} \arrow{e}
\node{0}\\
\node[2]{Y}
\end{diagram}\]
\noindent where ${g}{i}={f}$.
Let $i_{1}^{n}:X[1]^{n} \longrightarrow M(id)^{n}$ be canonical injection and $s^{n}:X[1]^{n-1} \longrightarrow Y^{n-1}$
such that $s^{n}={g^{n-1}}{i_{1}^{n-1}}$ for all $n\in \mathbb{Z}$ . Let u be the differential of the complex $M(id)$. Then we have the following diagram as follow.
\[\begin{diagram}
\node{{X^{n-1}}\oplus {X^{n-2}}} \arrow{e,t} {u^{n-2}} \arrow{s,r}
{g^{n-2}} \node{{X^{n}}\oplus {X^{n-1}}} \arrow{e,t} {u^{n-1}}
\arrow{s,r} {g^{n-1}} \node{{X^{n+1}}\oplus {X^{n}}} \arrow{s,r}
{g^{n}}\\
\node{Y^{n-2}} \arrow{e,t} {\gamma^{n-2}} \node{Y^{n-1}} \arrow{e,t} {\gamma^{n-1}} \node{Y^{n}}
\end{diagram}\]
${s^{n+1}}{\lambda^{n}}+{\gamma^{n-1}}{s^{n}}={g^{n}}{i_{1}^{n}}{\lambda^{n}}+{\gamma^{n-1}}{g^{n-1}}{i_{1}^{n-1}}
={g^{n}}{i_{1}^{n}}{\lambda^{n}}+{g^{n}}{u^{n-1}}{i_{1}^{n-1}}={g^{n}}({i_{1}^{n}}{\lambda^{n}}+{u^{n-1}}{i_{1}^{n-1}})
={g^{n}}{i^{n}}={f^{n}}$.
\end{proof}
\begin{lem} \label{1.19} Let $f:X \longrightarrow Y$ a chain homomorphism such that $Y$ is an
$\mathcal{X^{*}}$ complex and X is an $\mathcal{X}$-projective complex.
Then $f$ is a homotopic to zero. Moreover if a complex 
has a C($\mathcal{X}$-projective)-precover in $C(\mathcal{X^{*}})$, then such precovers are homotopic.
\end{lem}
\begin{proof} Let $id:Y \longrightarrow Y$ and the exact sequence $0 \longrightarrow Y[-1] \longrightarrow M(id)[-1]
\longrightarrow Y \longrightarrow 0$.
Since $X$ is an $\mathcal{X}-projective$ complex, we have the
following commutative diagram;
\[\begin{diagram}
\node{M(id)[-1]} \arrow{e,r} {\pi} \node{Y} \arrow{e} \node{0}\\
\node{X} \arrow{n,r,..} {g} \arrow{ne,t} {f}
\end{diagram}\]
\noindent where ${\pi}{g}={f}$. Let $\pi_{1}^{n}:M(id)[-1]^{n} \longrightarrow
Y[-1]^{n}$ be projection for all $n\in \mathbb{Z}$. Then if we take as $s^{n}={\pi_{1}^{n}}{g^{n}}$, then for
all $n\in \mathbb{Z}$,
${s^{n+1}}{\lambda^{n}}+{\gamma^{n-1}}{s^{n}}={f^{n}}$ where
$\lambda$ and $\gamma$ are boundary maps of the complexes of $X$ and
$Y$, respectively. So $f$ is homotopic to zero.
\end{proof}
\begin{lem} \label{1.24} Let $\mathcal{X}$ be extension closed. Let every R-module have an epic $\mathcal{X}$-projective 
precover  with kernel in $\mathcal{X}$. Then every bounded complex in C($\mathcal{X}^{*}$) has an epic exact C($\mathcal{X}-projective$)-precover (which is also in $\varepsilon_{\mathcal{X}-projective}$ if $(\mathcal{X},\mathcal{X^{\bot}})$ is a complete cotorsion pair) with kernel in C($\mathcal{X}^{*}$) (which is also in DG-$\mathcal{X}-projective$-injective=$(\varepsilon_{\mathcal{X}-projective})^{\bot}$). Moreover if $\mathcal{X}$ is a class of finitely generated modules, then every bounded complex in C($\mathcal{X}^{*}$) has an epic exact C($\mathcal{X}-f.g.projective$)-precover.   Thus every bounded $\mathcal{X}-(f.g.)projective$ complex in C($\mathcal{X}^{*}$) is exact (in $\varepsilon_{\mathcal{X}-projective}$ if $(\mathcal{X},\mathcal{X^{\bot}})$ is a cotorsion pair). 
\end{lem}
\begin{proof} Let $Y(n): ... \rightarrow 0 \rightarrow Y^{0} \rightarrow Y^{1} \rightarrow ... \rightarrow Y^{n} \rightarrow
0 \rightarrow...\in C(\mathcal{X}^{*})$.
We use induction on n. Let $n=0$, then we have the following
commutative diagram;
\[\begin{diagram}
\node{D(0):...} \arrow{e} \node{0} \arrow{e} \arrow{s} \node{P^{0}}
\arrow{e,t} {id} \arrow{s,r} {f^{0}} \node{P^{0}} \arrow{e}
\arrow{s} \node{0} \arrow{e}
\arrow{s} \arrow{e} \node{...} \\
\node{Y(0):...} \arrow{e} \node{0} \arrow{e} \node{Y^{0}} \arrow{e}
\node{0} \arrow{e} \node{0}
\arrow{e} \node{...}
\end{diagram}\]
\noindent where $P^{0}\rightarrow Y^{0}\rightarrow 0$ is an $\mathcal{X}$-projective precover in $\mathcal{X}$  with kernel in $\mathcal{X}$ since $\mathcal{X}$ is extension closed , $D(0)$ is exact and $Ker(D(0) \rightarrow Y(0)) \in
\mathcal{C(X^{*})}$.
\noindent We assume the following diagram which is commutative;
\[\begin{diagram}
\node{D(n):...0} \arrow{e} \node{P^{0}} \arrow{e,t} {\lambda^{0}}
\arrow{s,r} {f^{0}} \node{P^{0} \oplus P^{1}} \arrow{e,t}
{\lambda^{1}} \arrow{s,r} {(0,f^{1})} 
\node{...P^{n-1} \oplus P^{n}} \arrow{e,t} {\lambda_{1}^{n}}
\arrow{s,r} {(0,f^{n})} \node{P^{n}...}  \arrow{s} \\
\node{Y(n):...0} \arrow{e} \node{Y^{0}} \arrow{e,t} {a^{0}}
\node{Y^{1}} \arrow{e,t} {a^{1}} \node{...Y^{n}}
\arrow{e} \node{0...} 
\end{diagram}\]
\noindent where $\lambda_{1}^{n}$ is onto, $D(n)$ is an exact C($\mathcal{X}-projective$)-precover of $Y(n)$ such that  $Ker(D(n) \rightarrow Y(n)) \in
\mathcal{C(X^{*})}$ and the $P^{i}\rightarrow Y^{i}\rightarrow 0$ are $\mathcal{X}$-projective precovers in $\mathcal{X}$  with kernels in $\mathcal{X}$ for $1\leq i\leq n$.
Since $D(n)\rightarrow Y(n)\rightarrow 0$ and $\overline{P^{n+1}}\rightarrow \underline{Y^{n+1}}\rightarrow 0$ are C($\mathcal{X}-projective$)-precovers, we have the following commutative diagram
\[\begin{diagram}
\node{D(n)} \arrow{e,t} {s} \arrow{s} \node{\overline{P^{n+1}}}
\arrow{s} \\
\node{Y(n)} \arrow{e} \node{\underline{Y^{n+1}}}
\end{diagram}\]
\noindent Thus we have the diagram as follow:
\[\begin{diagram}
\node{D(n):...0} \arrow{e} \node{P^{0}} \arrow{e,t} {\lambda^{0}}
\arrow{s} \node{P^{0} \oplus P^{1}...} 
\arrow{s} \arrow{e,t} {\lambda^{n-1}} \node{P^{n-1}
\oplus P^{n}} \arrow{e,t} {\lambda_{1}^{n}}
\arrow{s,r} {s^{1}} \node{P^{n}...}  \arrow{s,r} {s^{2}}  \\
\node{\overline{P^{n+1}}:...0} \arrow{e} \node{0} \arrow{e} \node{0...}
\arrow{e} \node{P^{n+1}}
\arrow{e,t} {1} \node{P^{n+1}...} 
\end{diagram}\]
\noindent where ${s^{2}}{\lambda_{1}^{n}}={s^{1}}$ and
${s^{1}}{\lambda^{n-1}}=0$.
Moreover we see that ${f^{n+1}}{s^{1}}={a^{n}}{(0,f^{n})}$ and ${f^{n+1}}{s^{2}}=0$ by the following diagrams:
\[\begin{diagram}
\node{P^{n-1} \oplus P^{n}} \arrow{e,t} {s^{1}} \arrow{s,r}
{(0,f^{n})} \node{P^{n+1}}
\arrow{s,r} {f^{n+1}} \\
\node{Y^{n}} \arrow{e,t} {a^{n}} \node{Y^{n+1}} 
\end{diagram}\]
\[\begin{diagram}
\node{P^{n}} \arrow{e,t} {s^{2}} \arrow{s} \node{P^{n+1}}
\arrow{s,r} {f^{n+1}} \\
\node{0} \arrow{e} \node{Y^{n+1}} 
\end{diagram}\]
Let $\lambda^{n}(x,y)=(\lambda_{1}^{n}(x,y),s^{1}(x,y))$, \;
$\lambda_{1}^{n+1}(x,y)=s^{2}(x)-y$. Then we have the commutative diagram:
\[\begin{diagram}
\node{D(n+1):...}\arrow{e} \node{P^{0}...} 
\arrow{s,r} {f^{0}}  \arrow{e}
\node{P^{n-1} \oplus P^{n}}
\arrow{e,t} {\lambda^{n}}
\arrow{s,r} {(0,f^{n})} \node{P^{n} \oplus P^{n+1}} \arrow{e,t} {\lambda_{1}^{n+1}} \arrow{s,r} {(0,f^{n+1})} \node{P^{n+1}...}\arrow{s}\\
\node{Y(n+1):...}\arrow{e} \node{Y^{0}...} \arrow{e}
 \node{Y^{n}}
\arrow{e,t} {a^{n}} \node{Y^{n+1}}\arrow{e} \node{0...}
\end{diagram}\]
\noindent where  $Ker(D(n+1) \rightarrow
Y(n+1)) \in C(\mathcal{X}^{*})$ and since $\lambda_{1}^{n+1}$ is onto, $Im (\lambda^{n})=Ker(\lambda_{1}^{n+1})$, $Im (\lambda^{n-1})=Ker(\lambda^{n})$ and $D(n)$ is exact, $D(n+1)$ is exact. 
Therefore, $Y(n)$ has a C($\mathcal{X}$-projective) precover.
\end{proof}
\noindent We know that the direct (inverse) limit of exact complexes is also exact. Then we can give the following.
\begin{cor}\label{core}Let $(\mathcal{X},\mathcal{X^{\bot}})$ be a complete cotorsion pair. Then every bounded complex in C($\mathcal{X}^{*}$)  has an $\varepsilon_{\mathcal{X}-projective}$-precover.
\end{cor}
\begin{lem} \label{1.25} If $\mathcal{X}$ be extension closed and every R-module has a monic
$\mathcal{X}$-injective preenvelope with cokernel in $\mathcal{X}$, then every bounded complex in C($\mathcal{X}^{*}$)  has a monic exact C($\mathcal{X}$-injective)-preenvelope (which is also in $\varepsilon_{\mathcal{X}-injective}$ if $(\mathcal{X},\mathcal{X^{\bot}})$ is a complete cotorsion pair) with cokernel in C($\mathcal{X}^{*}$) (which is also in DG-$\mathcal{X}-injective$-projective=$^{\bot}(\varepsilon_{\mathcal{X}-injective})$). Moreover if $\mathcal{X}$ is a class of finitely generated modules, then every bounded complex in C($\mathcal{X}^{*}$)  has a monic exact C($\mathcal{X}-f.g.injective$)-preenvelope. Thus every bounded $\mathcal{X}-(f.g.)injective$ complex in C($\mathcal{X}^{*}$) is exact. 
\end{lem}
\begin{proof} Let $Y(n): ... \rightarrow 0 \rightarrow Y_{n} \rightarrow Y_{n-1} \rightarrow ... \rightarrow Y_{0} \rightarrow
0 \rightarrow...$.
We use induction on n. Let $n=0$, then we have the following
commutative diagram;
\[\begin{diagram}
\node{Y(0):...0} \arrow{e} \arrow{s} \node{0}
\arrow{e} \arrow{s} \node{Y_{0}} \arrow{e} \arrow{s} \node{0...} \\
\node{E(0):...0}  \arrow{e} \node{E_{0}}
\arrow{e,t} {id} \node{E_{0}} \arrow{e} \node{0...}  
\end{diagram}\]
\noindent where $0 \rightarrow Y_{0} \rightarrow E_{0}$ is a monic preenvelope in $\mathcal{X}$ with cokernel in $\mathcal{X}$ and thus $ E(0)$ is an exact preenvelope of $Y(0)$ with cokernel in $C(\mathcal{X}^{*})$.
We assume the following diagram which is commutative;
\[\begin{diagram}
\node{Y(n):...0} \arrow{e}  \node{0}  \arrow{e} \arrow{s}
\node{Y_{n}} \arrow{e,t} {a_{n}} \arrow{s,r} {(f_{n},0)} \node{...}
\arrow{e} \node{Y_{0}}
\arrow{e} \arrow{s,r} {f_{0}} \node{0}\\
\node{E(n):...0} \arrow{e} \node{E_{n}} \arrow{e,t} {\lambda_{n}^{1}}
\node{E_{n} \oplus E_{n-1}} \arrow{e,t} {\lambda_{n-1}} \node{...}
\arrow{e}
\node{E_{0}} \arrow{e} \node{0} 
\end{diagram}\]
\noindent where the $0\rightarrow Y_{i}\rightarrow E_{i}$ are $\mathcal{X}$-injective preenvelopes in $\mathcal{X}$  with cokernel in $\mathcal{X}$ for $1\leq i\leq n$, $E(n)$ is exact and cokernel $(Y(n)\rightarrow E(n))\in C(\mathcal{X}^{*})$.
Since $0\rightarrow Y_{n}\rightarrow E_{n}$ and $0\rightarrow \underline{Y_{n+1}}\rightarrow \overline{E_{n+1}} $ are C($\mathcal{X}$-injective)-preenvelopes, we have the following commutative diagram:
\[\begin{diagram}
\node{\underline{Y_{n+1}}} \arrow{e} \arrow{s} \node{Y(n)} \arrow{s} \\
\node{\overline{E_{n+1}}} \arrow{e} \node{E(n)}
\end{diagram}\]
\noindent Then we have the diagram as follow:
\[\begin{diagram}
\node{\overline{E_{n+1}}:...0} \arrow{e}  \node{E_{n+1}}  \arrow{e,t} {1}\arrow{s,r}{s_{n+1}}
\node{E_{n+1}} \arrow{e} \arrow{s,r} {s_{n}} \node{0} \arrow{e}\arrow{s}\node{...}\\
\node{E(n):...0} \arrow{e} \node{E_{n}} \arrow{e,t} {\lambda_{n}^{1}}
\node{E_{n} \oplus E_{n-1}} \arrow{e,t} {\lambda_{n-1}} \node{...}
\arrow{e}\node{...}  
\end{diagram}\]
\noindent where $s_{n}=\lambda_{n}^{1}s_{n+1}$ and $\lambda_{n-1}s_{n}=0$.
Moreover we see that $(f_{n},0){a_{n+1}}={s_{n}}{f_{n+1}}$ and ${\lambda_{n}^{1}}{s_{n+1}}={s_{n}}$ by the following diagrams:
\[\begin{diagram}
\node{Y_{n+1}} \arrow{e,t} {a_{n+1}} \arrow{s,r}
{f_{n+1}} \node{Y_{n}} \arrow{s,r} {(f_{n},0)}\\
\node{E_{n+1}} \arrow{e,t} {s_{n}} \node{E_{n} \oplus E_{n-1}} 
\end{diagram}\]
\[\begin{diagram}
\node{E_{n+1}} \arrow{e,t} {s_{n+1}} \arrow{s,r} {1}
\node{E_{n}} \arrow{s,r} {\lambda_{n}^{1}}\\
\node{E_{n+1}} \arrow{e,t} {s_{n}} \node{E_{n} \oplus E_{n-1}} 
\end{diagram}\]
Let $\lambda_{n+1}^{1}(x)=(x,-s_{n+1}(x))$, \;
$\lambda_{n}(x,y)=s_{n}(x)+\lambda^{1}_{n}(y)$. Then we have the following commutative diagram:
\[\begin{diagram}
\node{Y(n+1):...0} \arrow{e} \node{0} \arrow{e} \arrow{s}
\node{Y_{n+1}} \arrow{e,t} {a_{n+1}} \arrow{s,r} {(f_{n+1},0)}
\node{Y_{n}...} \arrow{e}  \arrow{s,r} {(f_{n},0)}
\node{Y_{0}...} 
\arrow{s,r} {f^{0}} \\
\node{E(n+1):...0} \arrow{e} \node{E_{n+1}} \arrow{e,t}
{\lambda_{n+1}^{1}} \node{E_{n+1} \oplus E_{n}} \arrow{e,t}
{\lambda_{n}} \node{E_{n} \oplus E_{n-1}...} \arrow{e} \node{E_{0}...} 
\end{diagram}\]
\noindent where $E(n+1)$ is exact and cokernel in $C(\mathcal{X}^{*})$.
Therefore, $Y(n)$ has a C($\mathcal{X}-(f.g.)injective$)-preenvelope.
\end{proof} 
\begin{cor}\label{cor}Let $(\mathcal{X},\mathcal{X^{\bot}})$ be a complete cotorsion pair. Then every bounded complex in C($\mathcal{X}^{*}$)  has an $\varepsilon_{\mathcal{X}-injective}$-preenvelope.
\end{cor}

\begin{thm}\label{tahire} Let $\mathcal{X}$ be closed under extensions. Let every R-module have a monic(epic)$\mathcal{X}$-injective(projective) preenvelope(precover) with cokernel (kernel)in $\mathcal{X}$. Then every
left (right) bounded complex in C($\mathcal{X}^{*}$) has a monic (epic) exact C($\mathcal{X}$-injective(projective)) preenvelope(precover).\\ 
 Moreover if $(\mathcal{X},\mathcal{X^{\bot}})$ is a complete cotorsion pair, then every
left (right) bounded complex in C($\mathcal{X}^{*}$) has a monic (epic)$\varepsilon_{\mathcal{X}-injective(projective)}$-pre-\\
envelope (precover).
 
\end{thm}
\begin{proof}Let $Y: ...\rightarrow 0 \rightarrow Y^{0} \rightarrow Y^{1}\rightarrow...$ and $E(n)$ be a C($\mathcal{X}$-injective) preenvelope of
$Y(n): ... \rightarrow 0 \rightarrow Y^{0} \rightarrow ...
\rightarrow Y^{n}\rightarrow 0 \rightarrow ...$.
Then $\underleftarrow{lim}Y(n)=Y$.
By Lemma \ref{1.25}, $Y(n)$ has a
C($\mathcal{X}$-injective) preenvelope $E(n)$ such that  $0 \rightarrow Y(n) \rightarrow E(n)$ is exact. Then
by Theorem 1.5.13 in \cite{4} $0 \rightarrow \underleftarrow{lim}Y(n)
\rightarrow \underleftarrow{lim}E(n)$ is exact with cokernel
$\underleftarrow{lim}\frac{E(n)}{Y(n)}\in C(\mathcal{X^{*}})$.
Since $Ext^{1}(\frac{A}{B},\underleftarrow{lim}E(n))=0$ where $\frac{A}{B}\in C(\mathcal{X^{*}})$, $\underleftarrow{lim}E(n)$ is an exact C($\mathcal{X}$-injective)preenvelope of Y. The other parts are also proved similarly.
\end{proof} 
\begin{cor}\label{gunes} i)If $\mathcal{X}$ is closed under direct sums and extensions and a special precovering and preenveloping class, then every left (right) bounded complex has a monic (epic) exact C($\mathcal{X}$-injective(projective)) preenvelope (precover) where $\mathcal{X}-inj$ and $\mathcal{X}-proj\subseteq \mathcal{X}$.\\
ii)If  $(\mathcal{X},\mathcal{X^{\bot}})$ is a complete cotorsion pair and C($\mathcal{X}^{*}$) is closed under inverse limit, then every  complex has a monic (epic) exact C($\mathcal{X}$-injective(projective)) preenvelope (which is in $\varepsilon_{\mathcal{X}-injective(projective)}$) whe-\\
re $\mathcal{X}-inj$ and $\mathcal{X}-proj\subseteq \mathcal{X}$.
\end{cor}
\begin{proof} Since every (left (right)bounded) complex has a monic (epic) (left(right) bounded) C($\mathcal{X}^{*}$) envelope (cover) under this conditions, by Theorem \ref{tahire}, Lemma \ref{1.24} and Lemma \ref{1.25} every (left (right)bounded) complex  has the requires, too.  
\end{proof} 
\begin{thm} \label{q} Let $\mathcal{X}$ be a class of finitely presented modules and be closed under extensions and direct limit. Let every R-module have a monic
$\mathcal{X}$-injective preenvelope with cokernel in $\mathcal{X}$. Then every
right bounded complex in C($\mathcal{X}^{*}$) has a monic exact C($\mathcal{X}$-f.g.injective) preenvelope. Moreover every right bounded $\mathcal{X}$-f.g.injective complex in C($\mathcal{X}^{*}$) is exact.
\end{thm}
\begin{proof} Let $Y: ... \rightarrow Y_{2} \rightarrow Y_{1} \rightarrow Y_{0} \rightarrow 0 \rightarrow
...$ and $E(n)$ be a C($\mathcal{X}$-f.g.injective) preenvelope of
$Y(n): ... \rightarrow 0 \rightarrow Y_{n} \rightarrow ...
\rightarrow Y_{1} \rightarrow Y_{0} \rightarrow 0 \rightarrow ...$.
Then $\underrightarrow{lim}Y(n)=Y$.
By Lemma \ref{1.25}, $Y(n)$ has a
C($\mathcal{X}$-f.g.injective) preenvelope $E(n)$ such that  $0 \rightarrow Y(n) \rightarrow E(n)$ is exact. Then
by Theorem 1.5.6 in \cite{4} $0 \rightarrow \underrightarrow{lim}Y(n)
\rightarrow \underrightarrow{lim}E(n)$ is exact with cokernel
$\underrightarrow{lim}\frac{E(n)}{Y(n)} \in C(\mathcal{X^{*}})$ .
Since $\mathcal{X}$ is a class of finitely presented modules, $Ext^{1}(\frac{A}{B},\underrightarrow{lim}E(n))=0$ where $\frac{A}{B}$\; is a finitely presented complex in $C(\mathcal{X^{*}})$ . So
$\underrightarrow{lim}E(n)$ is an exact C($\mathcal{X}$-f.g.injective)
preenvelope of Y.
\end{proof}
\begin{thm}\label{ali} Let $\mathcal{X}$ be a class of finitely presented modules and closed under extensions and direct limits and $C(\mathcal{X^{*}})$ be closed under inverse limits. Let every R-module have a monic
$\mathcal{X}$-injective preenvelope  with cokernel in $\mathcal{X}$. Then every
complex in C($\mathcal{X}^{*}$) has a monic exact C($\mathcal{X}$-f.g.injective) preenvelope. Then every \;$\mathcal{X}$-f.g.injective complex in C($\mathcal{X}^{*}$) is exact.
\end{thm}
\begin{proof} Let $Y: ... \rightarrow Y_{2} \rightarrow Y_{1} \rightarrow Y_{0} \rightarrow Y_{-1} \rightarrow
...$ and $Y(n): ... \rightarrow Y_{1} \rightarrow Y_{0} \rightarrow
... \rightarrow  Y_{-n} \rightarrow 0 \rightarrow ...$. Then
$\underleftarrow{lim}Y(n)=Y$.
\noindent By Theorem \ref{q}, $Y(n)$ has a
C($\mathcal{X}$-f.g.injective) preenvelope $D(n)$ such that $0
\rightarrow Y(n) \rightarrow D(n)$ is monic and $\frac{D(n)}{Y(n)}
\in C(\mathcal{X^{*}})$. So, $0 \rightarrow \underleftarrow{lim}Y(n)
\rightarrow \underleftarrow{lim}D(n)$ is monic with cokernel $\underleftarrow{lim}\frac{D(n)}{Y(n)} \in C(\mathcal{X^{*}})$ by Theorem 1.5.13 in \cite{4}.
\noindent Since \;$Ext^{1}(\frac{A}{B},\underleftarrow{lim}D(n))\cong
\underleftarrow{lim}Ext^{1}(\frac{A}{B},D(n))=0$ where $\frac{A}{B}$ is a f.g. complex in $C(\mathcal{X^{*}})$, $\underleftarrow{lim}D(n)$ is a
C($\mathcal{X}$-f.g.injective) preenvelope of Y.
\end{proof}
\begin{example}Let $\mathcal{X}$ be a class of $R$-modules closed under quotients, extensions and direct sums (so it is closed under direct limits since it is closed pure quotient) (for the existence of such classes, if $\mathcal{X}$ is a class of injective modules on a hereditary noetherian ring which is constructed in \cite{6}, then $\mathcal{X}$ is closed under quotients, extensions and direct limits and moreover if $\mathcal{X}$ is the class of min-injective modules and simple ideals of ring $R$ are projective, then it is closed under quotients, extensions and direct sums) and $A$ and $B$ are in $\mathcal{X}$ such that $\phi:A\rightarrow B$ is a homomorphism.  Then by Theorem 2.10 in \cite{1}, we have monic $\mathcal{X}-injective$-preenvelopes such that $f: A\rightarrow E_{A}$ and $g: B\rightarrow E_{B}$ with cokernels in $\mathcal{X}$. Then there exists a homomorphism $s:E_{A}\rightarrow E_{B} $ such that $g\phi =sf$. Using Lemma \ref{1.25} we can determine an exact C($\mathcal{X}-injective$)-preenvelope $E(1)$ of complex $Y(1)$ as follow:
\[\begin{diagram}
\node{Y(1):...0} \arrow{e} \arrow{s}\node{0} \arrow{e} \arrow{s} \node{A}
\arrow{e,t} {\phi}\arrow{s,r}{(f,0)} \node{B} \arrow{e} \arrow{s,r}{g} \node{0...} \\
\node{E(1):...0}  \arrow{e} \node{E_{A}}
\arrow{e,t} {\alpha} \node{E_{A}\oplus E_{B}} \arrow{e,t}{\beta} \node{E_{B}}\arrow{e} \node{0...} 
\end{diagram}\]
where $\alpha(x)=(x,-s(x))$ and $\beta(x,y)=s(x)+y$.
Then every left (right) bounded complex in $C(\mathcal{X^{*}})$ has a monic (epic) exact C(X-injective(pro-\\
jective)) preenvelope(precover) by Theorem \ref{tahire}. Moreover  if $C(\mathcal{X^{*}})$ is closed under inverse limit, by Theorem \ref{ali} every complex has a monic exact C($\mathcal{X}$-f.g.injective) preenvelope.
\end{example}
\begin{example}Let $\mathcal{X}$ be the class of finitely presented modules and $(\mathcal{X},\mathcal{X^{\bot}})$ be a cotorsion pair. Then $(\mathcal{X},\mathcal{X^{\bot}})$ is a complete cotorsion pair. If $C(\mathcal{X^{*}})$ is closed under inverse limits, then by Corollary \ref{gunes} every  complex has a monic (epic) exact C($\mathcal{X}$-injective(projective)) preenvelope (which is in $\varepsilon_{\mathcal{X}-injective(projective)}$) where $\mathcal{X}-inj$ and $\mathcal{X}-proj\subseteq \mathcal{X}$. Moreover by Theorem \ref{tahire} every
left (right) bounded complex in C($\mathcal{X}^{*}$) has a monic (epic)$\varepsilon_{\mathcal{X}-injective(projective)}$-preenvelope (precover) and by Theorem 3.10 every complex in $C(\mathcal{X^{*}})$ has a monic exact C(X-f.g.injective) preenvelope.
\end{example}
\begin{thm} \label{1.29} A complex $X$ is contained in a minimial $\mathcal{X}-injective$ complex $X'$.
\end{thm}
\begin{proof} We know that every complex has an
injective envelope. Let $S= \{A: X \subseteq A \subseteq E$ $and$
$A$ $\mathcal{X}- injective$ $complex\} \neq \emptyset$ and $S'$ be
a descending chain of $S$. We will show that $\cap_{A_{\alpha} \in
S'}\{A_{\alpha}: \alpha \in I\}$ is an $\mathcal{X}-injective$
complex. Using this pushout diagram we have the following diagram
where $C$ with $\mathcal{X}^{*}$-complex,
\[\begin{diagram}
\node{0} \arrow{e} \node{\cap A_{\alpha}} \arrow{e,t} {\beta}
\arrow{s,r} {\theta_{\alpha}} \node{Y} \arrow{e} \arrow{s,r} {\phi}
\node{C} \arrow{e}
\arrow{s} \node{0} \\
\node{0} \arrow{e} \node{A_{\alpha}} \arrow{e,t} {\gamma_{\alpha}}
\node{B_{\alpha}} \arrow{e} \node{C} \arrow{e}
\node{0} 
\end{diagram}\]
\noindent Then the bottom row is split exact. Therefore $0 \longrightarrow
\cap A_{\alpha}  \longrightarrow \cap B_{\alpha}$\\
$ \longrightarrow C\longrightarrow 0$ is split exact. We have the following diagram;
\[\begin{diagram}
\node{0} \arrow{e} \node{\cap A_{\alpha}} \arrow{e,t} {\beta}
\arrow{s} \node{Y} \arrow{e} \arrow{s,r} {\phi} \node{C} \arrow{e}
\arrow{s} \node{0} \\
\node{0} \arrow{e} \node{\cap A_{\alpha}} \arrow{e,t} {\gamma}
\node{\cap B_{\alpha}} \arrow{e} \node{C} \arrow{e}
\node{0} 
\end{diagram}\]
\noindent By five lemma $\phi$ is an isomorphism. Therefore $0\longrightarrow \cap A_{\alpha}  \longrightarrow Y$\\
$ \longrightarrow C \longrightarrow 0$ is split exact. So $S$ has a minimal element,
say $X'$.
\end{proof}



\begin{thebibliography}{9}
\bibitem{2}{Enochs,E.E.,Jenda, O.M.G. , and Xu, J., {\it Orthogonality in
the category of complexes}, Math. J. Okayama Univ., 38:25-46, (1996).}

\bibitem{4} {Enochs, E.E., Jenda, O.M.G., {\it Relative Homological Algebra}, de Gruyter Ex. Math. Volume 30,
 Walter de Gruyter and Co., Berlin (2000).}

\bibitem{5} {Gillespie, J., {\it The flat model structure on Ch(R)}, Trans Amer. Math. Soc.,356(8): 3369-3390 (2004). }

\bibitem{7} {Holm, H., J{\o}rgensen, P., {\it Covers, precovers and purity}, Illinois Journal of Math., \textbf{52(2)}, Number \textbf{2}, 691-703, (2008).}

\bibitem{1} {Mao, Lixin , Ding, Nanqing, {\it L-injective Hulls of Modules}, Bull. Austral. Math. Soc.,74:37-44, (2006).}


\bibitem{3} {Rotman, Joseph J., {\it An Introduction to Homological Algebra}, Springer, New York  (2009).}

\bibitem{6} {Stafford, J.T., Warfield, R.B., {\it Construction of Hereditary Noetherian Rings}, Proc. London Math. Soc. (3), 51:1-20, (1985).}


\end{thebibliography}
\end{document}